\documentclass[12pt,twoside, final]{amsart}
\usepackage{amsmath,amsthm,amscd,amsfonts,amssymb,enumerate}
\usepackage{graphicx}
\usepackage{color}
\usepackage[colorlinks]{hyperref}

\newtheorem{theorem}{Theorem}[section]
\newtheorem{proposition}[theorem]{Proposition}
\newtheorem{lemma}[theorem]{Lemma}

\theoremstyle{definition}

\newtheorem{exam}[theorem]{Example}
\theoremstyle{remark}
\newtheorem{remark}[theorem]{Remark}
\numberwithin{equation}{section} 

\DeclareMathOperator{\Sing}{Sing} 
 \DeclareMathOperator{\Nm}{Nm}
\DeclareMathOperator{\J}{J} \DeclareMathOperator{\Ker}{Ker}
\DeclareMathOperator{\Supp}{Supp}
\DeclareMathOperator{\Cliff}{Cliff} \DeclareMathOperator{\Pic}{Pic}

\begin{document}

\title[A Note on Projectivized Tangent Cone Quadrics ...]{
A Note
 on Projectivized Tangent Cone Quadrics of Rank $\leq 4$ in the Ideal of a Prym-Canonical Curve
 }

\author[Ali Bajravani]{Ali Bajravani}
\address[]{Department of Mathematics, Azarbaijan Shahid Madani University, Tabriz, I. R. Iran.\\
P. O. Box: 53751-71379}

\email{bajravani@azaruniv.edu}

\maketitle

\begin{abstract}
Throughout the paper, among other results, we give in theorem
\ref{theorem3} and proposition \ref{prepro1} a partial analogue of
theorem \ref{thm0} for projectivized tangent cone quadrics of rank
equal or less than $4$, for Prymians. During the lines of the paper
it would be seen that for an un-ramified double covering of a
general smooth tetragonal curve $X$ induced by a line bundle $\eta$
on $X$ with $\eta^2=0$, the Prym-canonical model of $X$
 is projectively normal in $\mathbf{P}(H^{0}(K_{X}\cdot\eta))$.
 Then we consider a genus $g=7$,
tetragonal curve $C$ which is birationally isomorphic to a plane
sextic curve $X$ with ordinary singularities. As byproduct of
theorem \ref{theorem3} and proposition \ref{prepro1}, we show that
the stable projectivized tangent cone quadrics with rank equal or
less than $4$ of an un-ramified double covering of $C$, generate the
space of quadrics in $\mathbf{P}(H^{0}(K_{C}\cdot\eta))$ containing
$K_{C}\cdot\eta$-model of
 $C$, where $\eta$ is a line bundle on $C$ with $\eta^2=0$, obtained in section $4$.\\
%\noindent{\small {\it AMS Classification}?:
\end{abstract}
%%%%%%%%%%%%%%%%%%%%%%%%%%%%%%%%%%
%%%%%%%%%%%%%%%%%%%%%%%%%%%%%%%%%%%%%%%%%%%
%%%%%%%%%%%%%%%%%%%%%%%%%%%%%%%%%%%%%%%%%%%%%
%\end{document}
\section{Introduction}

For an $\acute{e}$tale double covering $\pi:\tilde{X}\rightarrow X$
of smooth curves, it is naturally associated a principally polarized
abelian variety the so called Prym variety of $\pi$, which is
denoted by $\mathbb{P}(\pi)$, whose principal polarization is
induced twice by $\Theta_{\tilde{X}}$, the theta divisor of
$\tilde{X}$. While this P.P.A.V. enjoys from some interesting
properties analogous to Jacobians, it behaves differently in some
another properties. Usually in most cases these differences lead to
a rich geometry which provides wide areas of research. For example,
although the theory of Prym varieties is old enough and has been
studied variously by many well known mathematicians since decades
ago, but surprisingly an analogue of the well known Riemann
singularity theorem for Prymians has been given relatively lately,
by R. Smith and R. Varley in \cite{S-V1} and its complete analogue
has given recently by S. C. Martin in \cite{S-C-M}.

 Another useful and nice package in the
land of Jacobians of canonical curves, is the well known theorem
\ref{thm0}, proved by Andreotti-Mayer in \cite{A-M} and by G. Kempf
in \cite{K}.

\begin{theorem}{\label{thm0}
Let $X$ be a smooth projective curve of genus $g$ on an
algebraically closed field of characteristic zero and $\mid
D\mid=g^{1}_{g-1}$ a complete linear series of degree $g-1$ and
dimension $1$ on $X$. Consider the corresponding double point of
$\Theta_{X}$:
 $$\mathcal{O}_{X}(D)=\mathcal{O}_{X}(K_{X}-D)\in
\Theta_{sing}.$$
 Then the projectivized tangent cone to $\Theta_{X}$
at $\mathcal{O}_{X}(D)$ is a quadric of rank at least or equal to
$4$ containing the canonical model of $X$ which can be described as
the union of the linear span of divisors in $\mid
D\mid=g^{1}_{g-1}$. Moreover the quadric is of rank $3$ precisely
when $\mid 2D\mid=\mid K_{X}\mid$.\\
Conversely a quadric $Q$ of rank less than or equal to $4$, through
$X$ is a tangent cone to $\Theta_{X}$ if one of its rulings cuts out
a complete linear series of degree $g-1$ and dimension $1$ on $X$.

}
\end{theorem}
Although it is completely known in the literature that the projectivized tangent cone
at a double point a of Prym-Theta divisor  
of general Prym-Canonical curves is a quadric of rank $6$ rather 
than rank $4$, but it might be interesting to know:

\begin{itemize}
\item[$\bullet$]How is the effect of linear subspaces of a Prym quadric
tangent cone on $C$, when the quadric is of rank $4$ containing $C$?
\end{itemize}

Equivalently we look for an analogue of theorem \ref{thm0} for
Prymians. The genus $g=7$ case, the first case where the singular
locus of prym theta divisor is nonempty, is the first case that has
to be dealt with. We see that a projectivized tangent cone quadric
of rank equal or less than $4$ at a stable singularity of Prym-theta
divisor of an $\acute{e}$tale double covering $\tilde{X}\rightarrow
X$, through the Prym-canonical model of $X$, imposes a linear series
of degree $d$ such that $d\in \{g-3, g-2, g-1\}$. If $d\in \{g-2,
g-1\}$ then for each $g\geq 7$ the linear series is complete, while
it would be complete for $d=g-3$ too provided that $g=7$. A partial
converse to this result will be given in Proposition
\ref{prepro1}.\\
Then we proceed to provide an evidence for the above question.
Precisely we give an example of a curve admitting 
projectivized tangent cone quadrics of rank equal or less than $4$.
We would like to give an example, on which
% which is a
 %?thetragonal curve of genus?
%?$7$?, ?
not only a complete converse of theorem \ref{theorem3} is
valid, but also its rank $4$
Prym quadric
tangent cones generate the space of quadrics containing Prym-canonical model of $C$.
 A curve which is birationally equivalent to a plane sextic
with three number of ordinary singularities, might seem a candidate
for this aim. But because of technical reasons a curve of this type
with three collinear ordinary singularities is needed. Whereas the
existence of a plane sextic curve with three number of ordinary
double points is a well known fact, the existence of a plane sextic
curve with three number of collinear ordinary singularities needs an
actual proof. Such a proof will be given in Theorem \ref{thm1}.

H. Lange and E. Sernesi in \cite{L-S} have proved that any
Prym-canonical line bundle on a curve of Clifford index$\geq 3$ is
globally generated and very ample and moreover its Prym-canonical
model is projectively normal in the projective space of the
Prym-canonical differential forms. When this is no longer true for
an arbitrary tetragonal curve, it would be verified not only for
Prym-canonical model of the example, obtained in Theorem \ref{thm1},
in the projective space of its Prym-canonical differential forms.
but also for Prym-canonical model of general tetragonal curves in
their projective space of their Prym-canonical differential forms. See Lemma \ref{lem1}.

\vspace{-.45cm}
\section{Preliminaries and Notations}
%\begin{proof}{
%This is claimed and proved in \cite[page $4954$]{S-V2}. }
%\end{proof}
%%%%%%%%%%%%%%%%%%%%%%%%%%%%%%%%%%%%%%%%%%%%%%%
%%%%%%%%%%%%%%%%%%%%%%%%%%%%%%%%%%%%%%%%%%%%%%%%%

A nontrivial line bundle $\eta$ on an irreducible nonsingular
projective curve $X$ with $\eta^{2}=0$ gives rise to a double
covering $\pi: \tilde{X}\rightarrow X$ and vice versa. For a
nontrivial line bundle $\eta$ with $\eta^{2}=0$ we denote by
$\pi_{\eta}$ the map induced by $\eta$. The kernel of the norm map
of $\pi_{\eta}$ denoted by $\Nm(\pi_{\eta})$, which is a subset of
$\J(\tilde{X})$, turns to be the union of two irreducible isomorphic
components one of them containing zero. The component containing
zero, denoted by $\mathbb{P}(\pi_{\eta})$, is called the Prym
variety of the double covering $\pi_{\eta}$ and  consists of line
bundles $\tilde{\mathcal{H}}$ on $\tilde{X}$, with
$\tilde{\mathcal{H}}\in \Ker(\Nm)$ and $h^{0}(\tilde{\mathcal{H}})$
is an even number. The other one which is denoted by $Z_{1}$,
 consists of line bundles
$\tilde{\mathcal{H}}$ on $\tilde{X}$, with $\tilde{\mathcal{H}}\in
\Ker(\Nm)$ and $h^{0}(\tilde{\mathcal{H}})$ is an odd number.

The theta divisor $\Theta_{\tilde{X}}$ induces a principal
polarization on $\mathbb{P}(\pi_{\eta})$ and the Prym variety
$\mathbb{P}(\pi_{\eta})$ turns to be a principally polarized abelian
variety with principal polarization $\mathbb{E}(\pi_{\eta})$. In
terms of dimensions of global sections of the points in
$\J(\tilde{X})$, the principally polarized abelian variety
$\mathbb{P}(\pi_{\eta})$ can be described as follows:
$$\mathbb{P}(\pi_{\eta})=\{\tilde{\mathcal{L}}\in\J(\tilde{X})\mid\Nm(\tilde{\mathcal{L}})=K_{X}
\!\!\!\!\quad,\!\!\!\!\quad
 h^{0}(\tilde{\mathcal{L}})\equiv 0 \quad \!\!\!\! (mod \quad \!\!\!\! 2)\}.$$
The singular locus of $\mathbb{E}(\pi_{\eta})$ has a similar
description in these terms too:
$$\begin{array}{ccc}
\Sing(\mathbb{E}(\pi_{\eta}))&=&\!\!\!\!\!\!\!\!\!\!\!\!\!\!\!\!\!\!\!\!\!\!\!\!\!\!\!\!\!\!\!\!\!\!\!
\!\!\!\!\!\!\!\!\!\!\!\!\!\!\!\!\!\!\!\!\!\!\!\!\!\!\!\!\!\!\!
\{\tilde{\mathcal{L}}\in\mathbb{P}(\pi_{\eta})\mid
h^{0}(\tilde{\mathcal{L}})\geq 4\}\\
&\mathbf{\cup}& \{\tilde{\mathcal{L}}\in\mathbb{P}(\pi_{\eta})\mid
h^{0}(\tilde{\mathcal{L}})=2 \!\!\!\!\quad,\!\!\!\!\!\quad
T_{\tilde{\mathcal{L}}}(\mathbb{P}(\pi_{\eta}))\subset
TC_{\tilde{\mathcal{L}}}(\Theta_{\tilde{X}})\}
\end{array}$$
 where $TC_{\tilde{\mathcal{L}}}(\Theta_{\tilde{X}})$ denotes the tangent cone of
$\Theta_{\tilde{X}}$ at $\tilde{\mathcal{L}}$. The singular points
of $\mathbb{E}(\pi_{\eta})$ with $h^{0}(\tilde{\mathcal{L}})\geq 4$
 are called stable singularities and singularities belonging to the
 second set
are called exceptional ones. For standard notations and details of
the subject, see  of \cite[Chapter $14$]{C-L}.

If $\eta$ is a nontrivial line bundle on $X$ such that $\eta^{2}=0$,
then the line bundle $K_{X}\cdot\eta$ is called a Prym-canonical
line bundle on $X$, if $K_{X}\cdot\eta$ is globally generated and
very ample, the irreducible (possibly singular) curve
$\phi_{K_{X}\cdot\eta}(X)$ is a Prym-canonical curve, where
$$\phi_{K_{X}\cdot\eta}:X\rightarrow
\mathbf{P}(H^{0}(K_{X}\cdot\eta))$$ is the morphism defined by
global sections of $K_{X}\cdot\eta$. We will denote the curve
$\phi_{K_{X}\cdot\eta}(X)$ by $X_{\eta}$. A linear series
$g^{r}_{d}$ on $X$ gives rise to a same linear series on $X_{\eta}$
via $\phi_{K_{X}\cdot\eta}$ and vice versa. In the absence of any
confusion, we use a same symbol for both of these linear series on
$X$ or on $X_{\eta}$.

\begin{theorem}{\label{Martin thm}
Let $\pi:\tilde{X}\rightarrow X$ be an $\acute{e}$tale double
covering induced by a line bundle $\eta$ such that $\eta^{2}=0$.
Assume moreover that the line bundle $K_{X}\cdot\eta$ is very ample
and globally generated. Then the projectivized tangent cone of
$\mathbb{E}(\pi_{\eta})$ at a double point $\tilde{\mathcal{L}}$ is
a quadric hypersurface in $\mathbf{P}(H^{0}(K_{X}\cdot\eta))$
containing $X_{\eta}$ if and only if $\tilde{\mathcal{L}}$ is a
stable singularity with $h^{0}(\tilde{\mathcal{L}})=4$.
 }
\end{theorem}
\begin{proof}{
Let $\tilde{\mathcal{L}}\in \Sing(\mathbb{E}(\pi_{\eta}))$ be a
double point of $\mathbb{E}(\pi_{\eta})$. Consider that using
 \cite[Corollary
$6.2.5$]{S-C-M}, we have $\phi_{K_{X}\cdot\eta}(X)\subset
\mathbb{P}TC_{\tilde{\mathcal{L}}}(\mathbb{E}(\pi_{\eta}))$ if and
only if $h^{0}(\tilde{\mathcal{L}})\geq 4$.
 If $h^{0}(\tilde{\mathcal{L}})>4$ then one would have
$h^{0}(\tilde{\mathcal{L}})\geq 6$. Therefore by Riemann-Kempf singularity theorem
$\deg(\mathbb{P}TC_{\tilde{\mathcal{L}}}(\Theta_{\tilde{X}}))\geq
6$. Now since
$\mathbb{P}TC_{\tilde{\mathcal{L}}}(\mathbb{E}(\pi_{\eta}))=2\mathbb{P}TC_{\tilde{\mathcal{L}}}(\Theta_{\tilde{X}})\cdot\mathbf{P}(H^{0}(K_{X}\cdot\eta))$,
the hypersurface
$\mathbb{P}TC_{\tilde{\mathcal{L}}}(\mathbb{E}(\pi_{\eta}))$ would
be of degree at least $3$ and vice versa.

 }
\end{proof}
\begin{lemma}\label{l}
If $\tilde{p},\tilde{q}\in \tilde{X}$ and  $L_{\tilde{p},\tilde{q}}$
is the line in $\mathbf{P}(H^{0}(K_{\tilde{X}}))$ joining
$\tilde{p}$ to $\tilde{q}$ then
$\phi_{K_{X}\cdot\eta}(\pi(\tilde{p}))=\overline{p}=L_{\tilde{p},\tilde{q}}\cap
\mathbf{P}(H^{0}(K_{X}\cdot\eta))$.
%for which $\phi_{\eta}$ is the Prym canonical?
%embedding of $C$ in $\mathbf{P}^{g-2}$?.
\end{lemma}
\begin{proof}{
This is claimed and proved in \cite[page $4954$]{S-V2}. }
\end{proof}
%%%%%%%%%%%%%%%%%%%%%%%%%%%%%%%%%%%%%%%%%%%%%%%
%%%%%%%%%%%%%%%%%%%%%%%%%%%%%%%%%%%%%%%%%%%%%%%%%
\section{Rank $\leq 4$ Projectivized Tangent Cone Quadrics}

Assume that $F$ is a smooth projective curve of genus $g$ with a very ample Prym-canonical line bundle
$K_{F}\cdot \eta$.
% ?denote by $F_{\eta}$ and $\mathbb{E}_{F}$ its?
%$K_{F}\cdot \eta$-model and prym-theta divisor?, ?respectively?.
%For a?
%divisor $D_{12}$ as obtained in previous section?, ?one has?
%$h^{0}(D_{12})=6$ clearly?.
\begin{theorem}\label{theorem3}
Assume that $F$ is a smooth non-hyperelliptic projective curve of
genus $g$ with a very ample Prym-canonical line bundle $K_{F}\cdot
\eta$.
 Assume moreover that $\pi_{\eta}:\tilde{F}\rightarrow F$
is an $\acute{e}$tale double cover of $F$.
 Let $Q$ be a quadric of rank equal or less than
$4$ containing $F_{\eta}$ in $\mathbf{P}(H^{0}(K_{F}\cdot\eta))$.
Furthermore for
$\tilde{\mathcal{L}}\in\Sing(\mathbb{E}(\pi_{\eta}))$, assume that
$Q$ is the projectivized tangent cone of $\mathbb{E}(\pi_{\eta})$ at
$\tilde{\mathcal{L}}$. Then one of the rulings of $Q$ cuts a
$g^{1}_{d}$ on $F_{\eta}$ with $g-3 \leq d \leq g-1$ and
$2\mathcal{O}(g^{1}_{d}))\otimes\mathcal{L}=K_{F}$,
 for some line bundle $\mathcal{L}$ on $F$
 which $\mathcal{L}$ is of degree $0$, $2$ or $4$. Additionally, the
$g^{1}_{d}$ is complete when $d\in\{g-2,g-1\}$. It is complete in
the case $d=g-3$ as well, provided that $g=7$.
\end{theorem}
\begin{proof}{
Assume that $\pi_{\eta}:\tilde{F}\rightarrow F$ is an
$\acute{e}$tale double covering. If $Q=Q_{\tilde{\mathcal{L}}}$ is a
projectivized tangent cone of $\Sing(\mathbb{E}(\pi_{\eta}))$ at
$\tilde{\mathcal{L}}$ which is a quadric of rank equal or less than
$4$ containing $F_{\eta}$, then one of its rulings cuts a
$g^{1}_{d}$, on $F_{\eta}$. For a divisor $E\in \mid g^{1}_{d}
\mid$, using the geometric Riemann-Roth theorem and considering that
the linear space $<E>$ inside $Q_{\tilde{\mathcal{L}}}$ is of
codimension $2$ in $\mathbf{P}(H^{0}(K_{F}\cdot\eta))$, one obtains
that $h^{0}(K_{F}\cdot\eta-E)=2$. Since $\dim(<E>)=g-4$ one has
$d=\deg(E)\geq g-3$.\\
 Set $\Lambda=<E>$ and consider that
$\Lambda=\tilde{\Lambda}\cap\mathbf{P}(H^{0}(K_{F}\cdot\eta))$ where
$\tilde{\Lambda}=<\tilde{E}>$ for some $\tilde{E}\in
\tilde{F}^{3}_{2g-2}$ such that
$\tilde{\mathcal{L}}=\mathcal{O}(\tilde{E})$, see \cite{S-V1}.
 For $p\in F_{\eta}$ setting $\pi_{\eta}^{-1}(p)=\{\tilde{p}, \tilde{q}\}$ using
Lemma \ref{l}, we have $p\in \Supp(E)$ if and only if $\{\tilde{p},
\tilde{q}\} \subset \Supp(\tilde{E})$. This observation implies that
$d=\deg(E)\leq g-1$.
\\

\noindent To see the completeness of $g^{1}_{d}$ assume that
$g^{1}_{d}
\subset \mid D \mid$ for some divisor $D$ in $g^{1}_{d}$ and set $\tilde{D}=\pi_{\eta}^{*}(D)$.\\
Assume first that $d=g-1$: then the equality
$h^{0}(D\cdot\eta)=h^{0}(K_{F}\cdot\eta-D)$ together with the
geometric Riemann-Roth theorem imply that $h^{0}(D\cdot\eta)=2$. The
assumption $d=g-1$, implies that $2g^{1}_{d}=K_{F_{\eta}}$ and one
can see from this that $\pi_{\eta}^{*}(g^{1}_{d})\in
\Sing(\mathbb{E}(\pi_{\eta}))$. In fact  by Lemma \ref{l}, for each
divisor $\Gamma\in \mid g^{1}_{d}\mid$ the divisor
$\tilde{\Gamma}=\pi^{*}_{\eta}(\Gamma)$ is a divisor associated to a
global section of $\tilde{\mathcal{L}}$, so one has
$\pi_{\eta}^{*}(g^{1}_{d})=\tilde{\mathcal{L}}$,
 in this case. Therefore $\pi_{\eta}^{*}(g^{1}_{d})\in
\Sing(\mathbb{E}(\pi_{\eta}))$.\\
If $h^{0}(g^{1}_{d})> 2$ then
$h^{0}(\pi_{\eta}^{*}(g^{1}_{d}))=h^{0}(g^{1}_{d})+h^{0}(g^{1}_{d}\cdot\eta)>
4$. This implies that
$Q_{\pi_{\eta}^{*}(g^{1}_{d})}=\cup_{\tilde{D}\in \mid
\pi_{\eta}^{*}(g^{1}_{d}) \mid}<\tilde{D}>$ is a hypersurface at
least of degree $6$ in $\mathbf{P}(H^{0}(K_{\tilde{F}}))$. Therefore
the hypersurface
$Q=\frac{1}{2}[(Q_{\pi_{\eta}^{*}(g^{1}_{d})})\cdot\mathbf{P}(H^{0}(K_{F}\cdot\eta))]$
would be a hypersurface of degree at least three. This by equality
of the quadrics $Q$ and $Q_{\tilde{\mathcal{L}}}$ is absurd. This
implies that the $g^{1}_{g-1}$ is complete.

 If $d=g-2$ then $h^{0}(D\cdot\eta)=1$ and there exist $\tilde{p},
\tilde{q}\in \tilde{F}$ such that
$\mathcal{O}(\tilde{D})\otimes\mathcal{O}(\tilde{p}+\tilde{q})\in
\Sing(\mathbb{E}(\pi_{\eta}))$. In fact as in the previous case for
each divisor $\Gamma\in \mid g^{1}_{d}\mid$ there are points
$\tilde{p}, \tilde{q}\in\tilde{F}$ such that
$\tilde{\Gamma}=\pi^{*}_{\eta}(\Gamma)+(\tilde{p}+\tilde{q})$ is a
divisor associated to a global section of $\tilde{\mathcal{L}}$. So
 one has
$$\mid\mathcal{O}(\tilde{D})\otimes\mathcal{O}(\tilde{p}+\tilde{q})\mid=\mid\tilde{\mathcal{L}}\mid\in
\Sing(\mathbb{E}(\pi_{\eta})).$$
 Now the relations
$$4=h^{0}(\mathcal{O}(\tilde{D})\otimes\mathcal{O}(\tilde{p}+\tilde{q}))\geq
h^{0}(D)+h^{0}(D\cdot\eta)+h^{0}(\mathcal{O}(\tilde{p}+\tilde{q}))=h^{0}(D)+1+1$$
imply that $h^{0}(D)=2=h^{0}(g^{1}_{d})$.

\noindent Finally if $d=g-3$ then for $g=7$, if $h^{0}(D)>2$ then
$F$ has a $g^{r}_{4}$ with $r\geq2$. This by Clifford's theorem and
non-hyper ellipticity of $F$ is a contradiction.
\\

\noindent Consider moreover that since
$h^{0}(K_{F}-2\mathcal{O}(g^{1}_{d}))\geq 1$, the line bundle
$\mathcal{L}:=K_{F}-2\mathcal{O}(g^{1}_{d}))$ is a line bundle
satisfying $2\mathcal{O}(g^{1}_{d}))\otimes\mathcal{L}=K_{F}$. }
\end{proof}
 %%%%%%%%%%%%%%%%%%%%%%%%%%%%%%%%%%%%%%%
 %%%%%%%%%%%%%%%%%%%%%%%%%%%%%%%%%%%%%%%%%%%%%

\begin{proposition}\label{prepro1}
Assume that $F$ and the assumptions about it are as in Theorem
\ref{theorem3}.
 Let $Q$ be a quadric of rank equal or less than 
 $4$ containing $F_{\eta}$ such that one of its rulings cuts a
complete $g^{1}_{d}$ on $F_{\eta}$ with $g-2 \leq d \leq g-1$ and
$2\mathcal{O}(g^{1}_{d})\otimes\mathcal{L} =K_{F}$,
 for some line bundle $\mathcal{L}$
which $\mathcal{L}$ is of degree $0$ or $2$ on $F$. Then $Q$ is a
projectivized tangent cone of $\Sing(\mathbb{E}(\pi_{\eta}))$.
\end{proposition}
%%%%%%%%%%%%%%%%%%%%%%%%%%%%%%%%%%%%%%%%%%%%%%%%%%%%
\begin{proof}{
 Assume that $Q\in \mid \mathcal{I}_{F_{\eta}}(2) \mid$
is of rank equal or less than $4$ such that one of its rulings cuts
a complete $g^{1}_{d}$ with $g-2\leq d \leq g-1$ and
 $2\mathcal{O}(g^{1}_{d})\otimes\mathcal{L}
= K_{F}$. If $\mathcal{L}=\mathcal{O}_{F}(p_{1}+p_{2}+\cdots+p_{t})$
and
$\pi_{\eta}^{*}(\mathcal{L})=\mathcal{O}_{\tilde{F}}(\bar{p}_{1}+\bar{p}_{2}+\cdots+\bar{p}_{t}+\bar{q}_{1}+\bar{q}_{2}+\cdots+\bar{q}_{t})$,
where $\bar{p}_{i}$ and $\bar{q}_{i}$ are conjugate, then for a sub
divisor $\tilde{D}_{1}$ of
$\tilde{D}=\bar{p}_{1}+\bar{p}_{2}+\cdots+\bar{p}_{t}+\bar{q}_{1}+\bar{q}_{2}+\cdots+\bar{q}_{t}$
which is of degree $\frac{1}{2}\deg(\tilde{D})$ and no two points of
its support are conjugate, setting
$$\frac{1}{2}(\pi_{\eta}^{*}\mathcal{L}):=\mathcal{O}_{\tilde{F}}(\tilde{D}_{1})
\quad , \quad
\tilde{\mathcal{L}}=\pi_{\eta}^{*}(\mathcal{O}(g^{1}_{d}))\otimes\frac{1}{2}(\pi_{\eta}^{*}\mathcal{L}).$$
one has $\Nm(\tilde{\mathcal{L}})=K_{F}$. This reads to say that
$\tilde{\mathcal{L}}\in \Ker(\Nm)=\mathbb{P}(\pi_{\eta})\cup Z_{1}$
where $Z_{1}$ is the isomorphic copy of $\mathbb{P}(\pi_{\eta})$
which we already introduced in backgrounds.

Consider the relations:
$$\begin{array}{ccc}
h^{0}(\tilde{\mathcal{L}})&=&\!\!\!
h^{0}(\pi_{\eta}^{*}(\mathcal{O}(g^{1}_{d}))+\frac{1}{2}(\pi_{\eta}^{*}\mathcal{L}))\geq
h^{0}(\pi_{\eta}^{*}(\mathcal{O}(g^{1}_{d})))+h^{0}(\frac{1}{2}(\pi_{\eta}^{*}\mathcal{L}))\\
\\
&=&\!\!\!\!\!\!\!\!\!\!\!\!\!\!\!\!\!\!\!\!\!\!\!\!\!\!\!\!\!\!\!\!\!\!\!\!\!\!\!\!\!\!\!h^{0}(\mathcal{O}(g^{1}_{d}))+
h^{0}(\mathcal{O}(g^{1}_{d})\cdot\eta)+h^{0}(\frac{1}{2}(\pi_{\eta}^{*}\mathcal{L})).
\end{array}$$
If $d=g-1$ then
$h^{0}(\mathcal{O}(g^{1}_{d}))=h^{0}(\mathcal{O}(g^{1}_{d})\cdot\eta)=2$
and $\mathcal{L}$ has to be equal to $\mathcal{O}_{F}$. Therefore
$h^{0}(\tilde{\mathcal{L}})=4$ and so $\tilde{\mathcal{L}}\in
\Sing(\mathbb{E}(\pi_{\eta}))$.
 Moreover $Q_{\tilde{\mathcal{L}}}=Q$ and this implies that $Q$ is a projectivized tangent cone in this case.

In the case of $d\!=\!g-2$ one has
$h^{0}(\mathcal{O}(g^{1}_{d})\cdot\eta)=1$ which implies that
$$\begin{array}{ccc}
\dim(\mid\pi_{\eta}^{*}(\mathcal{O}(g^{1}_{d}))\mid
)&=&\!\!\!\!\!\!\!\!\!\!\!\!\!\!\!\!\!\!\!\!\!\!\!\!\!\!\!\!\!\!\!\!\!\!\!\!\!\!\!\!\!\!
h^{0}(\pi_{\eta}^{*}(\mathcal{O}(g^{1}_{d})))-1\\
&=&h^{0}(\mathcal{O}(g^{1}_{d}))+h^{0}(\mathcal{O}(g^{1}_{d})\cdot\eta)-1=2.
\end{array}$$
This reads to say that taking a global section $\sigma$ of
$\pi_{\eta}^{*}(\mathcal{O}(g^{1}_{d}))$ and considering its
associated divisor,  $\tilde{D}$, one has
$\dim(\mid\tilde{D}\mid)=2$. From this it can be seen that taking a
global section $\gamma$ of $\tilde{\mathcal{L}}$ and considering its
associated divisor $\tilde{B}$, there exists a divisor $\tilde{E}\in
\mid\tilde{B}\mid$ such that $h^{0}(\tilde{E})=4$. In fact for a
point $p$ in the support of a divisor associated to a global section
of the line bundle $\frac{1}{2}(\pi_{\eta}^{*}\mathcal{L})$, assume
that $\tilde{D}_{1}\in \mid\tilde{D}\mid$ is a divisor such that
$p\in\Supp(\tilde{D}_{1})$. Now set
$\bar{D}_{1}:=\tilde{D}_{1}+p+q=\bar{M}+2p+q$ for some divisor
$\bar{M}$ on $\tilde{F}$ and some point $q\in\tilde{F}$ such that
$p+q\in\frac{1}{2}(\pi_{\eta}^{*}\mathcal{L})$. Consider that
$\mathcal{O}(\bar{D}_{1})\in \mid\tilde{\mathcal{L}}\mid$ and
$$\begin{array}{ccc}
\dim(<\bar{D}_{1}>)&=&\!\!\!\!\!\!\!\!\!\!\!\!\!\!\!\!\!\!\!\!\!\!\!\!\!\!\dim(<\bar{M}+p+q>)\\
&=&\dim(<\bar{M}+p>)+1=2g-6.
\end{array}$$ This equivalently reads to say that $h^{0}(\bar{D}_{1})=4$
which implies that $h^{0}(\tilde{\mathcal{L}})=4$. Therefore
$\tilde{\mathcal{L}}\in \Sing(\mathbb{E}(\pi_{\eta}))$ which
finishes the proof.

}
\end{proof}
%%%%%%%%%%%%%%%%%%%%%%%%%%%%%%%%%%%
%%%%%%%%%%%%%%%%%%%%%%%%%%%%%%%%%%%%%%%%%%%%%%%%%
\section{2-Normality of General Tetragonal Curves and an Example of Prym Tetragonal Curve}
Let $X$ be an irreducible plane sextic curve with $\bar{x}$,
$\bar{y}$ and $\bar{z}$ as its nodes or double points. Assume
moreover that $\bar{x}$, $\bar{y}$ and $\bar{z}$ are collinear. Let
$i:C\rightarrow X$ be its normalization. Notice that by genus
formula for plane curves, $X$, and consequently $C$, is of genus
$7$. Now on the curve $C$ consider the linear series $|5H-\Delta|$
for which $H=i^{*}(\mathcal{O}_{X}(1))$, $\Delta =
x_{1}+x_{2}+y_{1}+y_{2}+z_{1}+z_{2}$ with $i^{-1}(\bar{x})=\{x_{1},
x_{2}\}$, $i^{-1}(\bar{y})=\{y_{1}, y_{2}\}$ and
$i^{-1}(\bar{z})=\{z_{1}, z_{2}\}$. Notice that $K_{C}=3H-\Delta$
and $\deg(5H-\Delta)=24$. Now take a divisor $D_{12}\in C^{(12)}$
such that $2D_{12}$ belongs to $|5H-\Delta|$ and set
$\eta:=2H-D_{12}$. Trivially
 $\eta^{2}=4H-2D_{12}\sim 4H-(5H-\Delta)=\Delta-H\sim O$ %This means that the line bundle $\eta$ is a theta?
 %characteristic,
 where the last equality holds because the points $\bar{x}$, $\bar{y}$ and
$\bar{z}$ are collinear. Notice moreover that the lines passing
through one of the singularities of the curve $X$ define a base
point free $g^{1}_{4}$ on $X$ and $i^{*}(g^{1}_{4})$ is a base point
free $g^{1}_{4}$ on $C$. This implies that $C$ is an irreducible,
nonsingular, tetragonal curve of genus $7$ with three number of base
point free $g^{1}_{4}$'s.

\noindent Next we verify the existence of a curve $C$ admitting a
$D_{12}$ with mentioned properties:
\begin{theorem}{\label{thm1}
There exists a curve $C$ admitting a divisor $D_{12}$ with the
mentioned properties and admitting a globally generated and very
ample prym-canonical line bundle $K_{C}\cdot\eta$. }
\end{theorem}
\begin{proof}{
 Let $Q_{1}$ and $Q_{2}$ be quadrics in $\mathbf{P}^{2}$ tangent
to each other exactly in one point. Bezout's theorem implies that
they cut each other in two extra points $\bar{x}$ and $\bar{y}$ with
multiplicity one. Consider the line $l$ passing through $\bar{x}$
and $\bar{y}$. Since the tangent variety of $Q_{1}$, (resp. $Q_{2}$)
fills up all the surface $\mathbf{P}^{2}$, an arbitrary point
$\bar{z}\in l$ lies on at least a tangent line of $Q_{1}$, (resp.
$Q_{2}$).
 Therefore for an arbitrary point $\bar{z}$ on $l$ there
are a couple of lines $L_{1}$ and $L_{2}$ passing through $\bar{z}$
such that $L_{1}$ is tangent to $Q_{1}$ and $L_{2}$ is tangent to
$Q_{2}$. Then with this assignments, the reducible curve $X$ defined
by the polynomial $h=Q_{1}Q_{2}L_{1}L_{2}$ is a curve of degree six
which has three collinear ordinary singularities.

\noindent Denote by $t$ the point where $Q_{1}$ and $Q_{2}$ are
tangent to each other.
 A computation shows that there are infinitely many
quadrics through $\bar{x}$, $\bar{y}$ and $t$ such that each of
these quadrics has the same tangent line at the point $t$. In fact,
quadrics in $\mathbf{P}^{2}$ passing through the points $p=(1:0:0)$,
$q=(0:1:0)$ and $r=(0:0:1)$, are given by
$b_{0}x_{0}x_{1}+b_{1}x_{0}x_{2}+b_{2}x_{1}x_{2}=0$. These quadrics
have the line $b_{0}x_{0}+b_{1}x_{2}=0$ as their tangent line at the
point $p$. These imply that for fixed $d_{0}, d_{2}$ the infinitely
many quadrics $b_{0}x_{0}x_{1}+b_{1}x_{0}x_{2}+b_{2}x_{1}x_{2}=0$
passe through $p$, $q$, $r$ and are tangent to each other at the
point $p$.\\
Choose a couple of quadrics $\bar{Q}_{1}, \bar{Q}_{2}$ passing
through $\bar{x}$, $\bar{y}$, $t$, tangent to each other at $t$ and
distinct with $Q_{1}$ and $Q_{2}$ respectively. As in the sextic
$h$, there are lines $\bar{L}_{1}$ and $\bar{L}_{2}$ distinct
 from $L_{1}$ and $L_{2}$ respectively, passing through $\bar{z}$ and
are tangent to $\bar{Q}_{1}$ and $\bar{Q}_{2}$ respectively. Again
the curve $k=\bar{Q}_{1}\bar{Q}_{2}\bar{L}_{1}\bar{L}_{2}$ is a
reducible plane sextic having three collinear points $\bar{x}$,
$\bar{y}$ and $\bar{z}$ as its ordinary singularities. Now since $h$
and $k$ has no common irreducible component, Bertini's theorem
implies that $X$, a general member of the pencil generated by $h$
and $k$, is an irreducible plane sextic having three collinear
points as its ordinary singularities. Choosing the normalization of
$X$ gives the desired curve $C$.

\noindent To show existence of a $D_{12}$ with desirable properties,
take a plane quintic $T$ with three number of nodes $p_{1}$,
$p_{2}$, $p_{3}$ and passing through three distinct prescribed
collinear points $\bar{x}$, $\bar{y}$ and $\bar{z}$. Choose nine
extra points $p_{4},\cdots, p_{12}$ on $T$. Notice that passing
through points $p_{1},\cdots, p_{12}$, being tangent to $T$ at the
points $p_{4},\cdots, p_{12}$ and having three collinear points
$\bar{x}$, $\bar{y}$, $\bar{z}$ as only singularities,
 impose at most $24$ conditions on
the space of plane sextics. Since the space of plane sextics is of
dimension $27$, there are plane sextics $X$, passing through the
points $p_{1},\cdots, p_{12}$, being tangent to $T$ at the points
$p_{4},\cdots, p_{12}$ and having three collinear points $\bar{x}$,
$\bar{y}$, $\bar{z}$ as only singularities. On such a curve $X$,
setting $X_{12}:=p_{1}+\cdots+ p_{12}$ one obtains the desired
$D_{12}$.

 \noindent It can be seen easily that any Prym-canonical line bundle
 on a non-hyperelliptic curve is globally generated. See
 \cite[Lemma $2.1$]{L-S}.

\noindent To see very ampleness of a prym-canonical line bundle on
$C$, consider that the prym-canonical line bundle $K_{C}\cdot\eta$
with $\eta=2H-D_{12}$ is very ample. In fact the proof of Lemma
\ref{lem1} implies that in lack of very ampleness of
$K_{C}\cdot\eta$, there would exist two another points $z, w\in X$
such that
$$x+y\nsim z+w \quad , \quad 2x+2y\sim 2z+2w \in  g^{1}_{4}.$$
This means that taking $L$, the line passing through singularities
of $X$, there exists a line $\acute{L}$ other than $L$ that is
tangent to $X$ in two points $\alpha$ and $\beta$ distinct with $x,
y$ and $z$. But this is impossible for a general curve of type $X$.
}
\end{proof}
%%%%%%%%%%%%%%%%%%%%%%%%%%%%%%%%%%%%%
%%%%%%%%%%%%%%%%%%%%%%%%%%%%%%%%%%%%%%%%%%%%%%%%%%%%%%%%%%%%%%%%%%%%%%%%%%%%%%%%%%%%%%%%%%%%%%%%%%%%%%%%%%%%%%%%%%%%%%%%%%%
%%%%%%%%%%%%%%%%%%%%%%%%%%%%%%%%%%%%%%%%%%%%%
More than what we saw in Theorem \ref{thm1} and at least as an
independent interest, one can prove that any Prym-canonical line
bundle on a general tetragonal curve
is very ample and the Prym-canonical model of this line bundle is projectively normal:\\
Assume for a moment that the line bundle $K_{X}\cdot\eta$ is very
ample. Then $X_{\eta}$, the Prym-Canonical model of $X$ in
$\mathbf{P}(H^{0}(K_{X}\cdot\eta))$, is $2$-normal, namely the map
$$H^{0}(\mathbf{P}(H^{0}(K_{X}\cdot\eta)), \mathcal{O}_{\mathbf{P}(H^{0}(K_{X}\cdot\eta))}(2)) 
\rightarrow H^{0}(X_{\eta}, \mathcal{O}_{X_{\eta}}(2))$$
is surjective. In fact by \cite{G-L-P}, a curve $X$ is $n$-regular
if and only if, it is $(n-1)$-normal and the line bundle
$\mathcal{O}_{X}(n-2)$ is non-special
%%%%%%%%%%%%%%%%%%%%%%%%%%%%%%%%%%%%%%%%%%%%%%%%%%%%%%%%%%%%%%%%%%%%%%%%%%%%%%%%
Therefore to prove $2$-normality of $X_{\eta}$, it is enough to
prove its $3$-regularity, which means that its sheaf of ideals,
$\mathcal{I}_{X_{\eta}}$, is 3-regular. Moreover for $n\geq 1$,
using the exact sequence
%%%%%%%%%%%%%%%%%%%%%%%%%%%%%%%%%%%%%%%%%%%%%%%%%%%%%%%%%%%%%%%%%%%%%%%%%%%%%%%%%%%%%%%%
$$0\rightarrow
\mathcal{I}_{X_{\eta}}\rightarrow
\mathcal{O}_{\mathbf{P}(H^{0}(K_{X}\cdot\eta))}\rightarrow
\mathcal{O}_{X_{\eta}}\rightarrow 0,$$
 it is enough to prove that
the sheaf $\mathcal{O}_{X_{\eta}}$ is $2$-regular. Trivially
$H^{i}(X_{\eta},\mathcal{O}_{X_{\eta}}(2-i))=0$ for $i\geq 2$ by
Grothendieck's vanishing theorem. It remains to prove that
$H^{1}(X_{\eta},\mathcal{O}_{X_{\eta}}(1))=0$. To see this, consider
the isomorphisms
$$H^{1}(X_{\eta},\mathcal{O}_{X_{\eta}}(1))\cong H^{1}(X, K_{X}\cdot\eta)\cong (H^{0}(X, K_{X}-K_{X}\cdot\eta))^{\vee}=(H^{0}(X,\eta))^{\vee}$$
together with $H^{0}(X,\eta)=0$. These finish, $2$-normality of
$X_{\eta}$.

The discussion just has been done, proves projective normality of
curves $X$ with $\Cliff(X)\geq 3$, in which case any prym-canonical
line bundle $K_{X}\cdot\eta$ is very ample by \cite{L-S}. In the
case $\Cliff(X)=2$, it can happen that the line bundle
$K_{X}\cdot\eta$ is not very ample for special tetragonal curves
$X$. But it can be proved that for a general tethragonal curve $X$,
any prym-canonical line bundle $K_{X}\cdot\eta$ is very ample. In
fact,
 %?by proof of \cite[Lemma 2.1]{L-S}?,?
  an equality
$$h^{0}(K_{X}\cdot\eta(-x-y))=h^{0}(K_{X}\cdot\eta)-1$$
for some points $x, y$ on $X$, implies that there exist another
points $z, w\in X$ such that
$$x+y\nsim z+w \quad , \quad 2x+2y\sim 2z+2w \in  g^{1}_{4}.$$
This by \cite{D} is absurd for a general tetragonal curve.

\noindent Summarizing we have proved:
\begin{theorem}{\label{lem1}
Assume that $X$ is
 a general smooth tetragonal
curve of genus $g$ and $\tilde{X}\rightarrow X$ an etale double
covering of $X$ induced by $\sigma \in\Pic(X)$ with $\sigma^2=0$.
Then $X_{\sigma}$, the Prym-canonical model of $X$ in
$\mathbf{P}(H^{0}(K_{X}\cdot\sigma))$, is projectively normal in
$\mathbf{P}(H^{0}(K_{X}\cdot\sigma))$.

 }
\end{theorem}
%%%%%%%%%%%%%%%%%%%%%%%%%%%%%%%%%%%%%%%%
%%%%%%%%%%%%%%%%%%%%%%%%%%%%%%%%%%%%%%%%%%%%%%%
%%%%%%%%%%%%%%%%%%%%%%%%%%%%%%%%%%%%%%%%%%
\section{Projectivized Tangent Cone Quadrics of $C$ generate the Space of $H^{0}(\mathcal{I}_{C_{\eta}}(2))$}

Everywhere in this paper by $C$ we mean the tetragonal curve
obtained in Theorem \ref{thm1}. Moreover by $\eta$ we mean the
$2$-torsion line bundle obtained there in the rest of the paper.
\begin{theorem}{\label{thm5}
 Let $C$ and $\eta$ be the curve and the line bundle obtained in Theorem
 \ref{thm1}.
 Assume moreover that the double covering induced by $\eta$ is an
$\acute{e}$tale. Then a quadric $Q\in\mathcal{I}_{C_{\eta}}(2)$ of
rank equal or less than $4$ is a prjectivized tangent cone if and
only if one of its rulings cuts a complete $g^{1}_{d}$ with
$d\in\{g-3, g-2, g-1 \}$ and
$2\mathcal{O}(g^{1}_{d})\otimes\mathcal{L} =K_{C}$, for some line
bundle $\mathcal{L}$ on $C$ which is of degree $0$, $2$ or $4$.

}
\end{theorem}
\begin{proof}{
If a quadric $Q\in\mathcal{I}_{C_{\eta}}(2)$ is a projectivized
tangent cone of $\mathbb{E}(\pi_{\eta})$, then
 since $g(C)=7$ one of its rulings cuts the desired complete linear series,  by Theorem \ref{theorem3}.

\noindent Conversely if one of the rulings of a quadric
$Q\in\mathcal{I}_{C_{\eta}}(2)$ of rank equal or less than $4$ cuts
a complete $g^{1}_{d}$ with prescribed conditions, then Proposition
\ref{prepro1} implies that $Q$ is a projectivized tangent cone of
$\mathbb{E}(\pi_{\eta})$ provided that $d\in \{g-2, g-1 \}$. If
$d=g-3$, then one obtains a $g^{1}_{4}$ on $C$. By Martens-Mumford's
theorem there are only finitely many $g^{1}_{4}$'s on $C$. In fact
the pencils of lines through $\bar{x}$ through $\bar{y}$ or through
$\bar{z}$ cut out three $g^{1}_{4}$'s on $C$ and one can see that a
$g^{1}_{4}$ on $C$ is one of these pencils. But it is easy to see
that the rulings generated by divisors in these $g^{1}_{4}$'s are at
most of dimension $2$ and therefore these rulings can not sweep out
a quadric. In fact for $D\in \mid H-(x_{1}+x_{2})\mid$ one has:
$$\begin{array}{ccc}
\dim(<D>)&=&\!\!\!\!\!\!\!\!\!\!\!\!\!\!\!\!\!\!\!\!\!\!\!\!\!\!\!\!\!\!\!\!\!\!\!\!\!\!\!\!\!\!\!\!\!\!\!\!
\!\!\!\!\!\!\!\!\!\!\!\!\!\!\!\!\!\!\!\!\!\!\!\!5-h^{0}(K_{C}\cdot\eta-g^{1}_{4})\\
&=&5-h^{0}(3H-\Delta+2H-D_{12}-(H-(x_{1}+x_{2})))\\
&=& 5-h^{0}(D_{12}-H+x_{1}+x_{2}) \leq 5-h^{0}(D_{12}-H).
\end{array}$$

\noindent Since $C$ is non-hyperelliptic, the Clifford's theorem
asserts that $h^{0}(D_{12}-H)\leq 3$. Consider now that multiplying
by $H$ gives the following exact sequence:
$$0\rightarrow H^{0}(D_{12}-H)\rightarrow H^{0}(D_{12})\rightarrow H^{0}(H)$$
which implies that $h^{0}(D_{12})-h^{0}(D_{12}-H)\leq h^{0}(H)=3$.
Therefore one has
 $h^{0}(D_{12}-H)\geq 3$. Summarizing one has $h^{0}(D_{12}-H)=3$
 and so $\dim(<D>)\leq 5-h^{0}(D_{12}-H)=2$.\\
 These imply that $d$ can not be equal to $g-3=4$ because the linear spaces
  inside $Q$ generated by divisors in $g^{1}_{d}$
 have to sweep out the quadric itself.
}
\end{proof}
%\end{document}?
%%%%%%%%%%%%%%%%%%%%%%%%%%%%%%%%%%%%%%%%%%%%%%%%%%%%%%%%%%%%%%%%%%%%%%%%%%%%%%%%%%%%%%%%%%%%%%%%%%%
Now, in order to give an application of Theorem \ref{thm5}, we
describe $W^{1}_{5}$ and $W^{1}_{6}$ on $C$:

\begin{exam}\label{example1}{
\textbf{(i) $g^{1}_{5}$'s on $C$:}

%\end{document}?
 By Mumford-Martens theorem and considering that $C$ is not hyperelliptic nor trigonal and nor a smooth plane
 quintic,
 one has $\dim (W^{1}_{5})\leq 5-2-2=1$.
For each $p\in X-\{\bar{x}, \bar{y}, \bar{z}\}$, the lines
 passing through $p$ cut out a pencil of degree $5$ on $X$ as well as do
 the quadrics through $\bar{x}, \bar{y}, \bar{z}$ and $p$. These pencils give rise to pencils of the same kind on
 $C$ via pulling them back to $C$ by the normalization map, $i$. Consider moreover that the only way for a one dimensional sub vector space $V$, of
 quadrics in $\mathbf{P}^2$ to cut a $g^{1}_{5}$ on $X$ is that each member of $V$ has to pass
 through $\bar{x}, \bar{y}, \bar{z}$ and $p$.
 The cubiques in $\mathbf{P}^2$ can not cut a
 $g^{1}_{5}$ on $X$. Generally picking $6d-11$ points $p_{1}, p_{2},...,p_{6d-11}$ fixed on
 $X$, the hypersurfaces of degree $d$ in
 $\mathbf{P}^2$ passing through $\bar{x}, \bar{y}, \bar{z}$ and the chosen  $6d-11$ points $p_{1}, p_{2},...,p_{6d-11}$ will cut
 a $g^{r}_{5}$ on $X$. For $d\geq 4$ we have $r\geq 2$ and therefore
hypersurfaces of degree $d\geq 4$ won't cut a $g^{1}_{5}$. Therefore
$\dim (W^{1}_{5})=1$ and  $g^{1}_{5}$'s are cut on $X$ by lines or
quadrics of $\mathbf{P}^2$. Now these pencils pulled back via $i$,
are the only $g^{1}_{5}$'s on $C$.\\

\noindent If $g^{1}_{5}=H-p$ for some $p\in C$ then $K_{C}\cdot\eta
- g^{1}_{5}=D_{12}-H+p$. Consider that $h^{0}(D_{12})=6$ and by
proof of Proposition \ref{prepro1}, one has $h^{0}(D_{12}-H)=3$.
 Therefore $h^{0}(D_{12}-H-p)=3$ if $p$ belongs to
the base locus of $\mid D_{12}-H \mid$, and $h^{0}(D_{12}-H-p)=2$
otherwise. Taking the exact sequence
$$0\rightarrow H^{0}(D_{12}-H-p)\rightarrow H^{0}(D_{12}-H)\rightarrow H^{0}(H)\rightarrow 0$$
it is routine to see that $h^{0}(K_{C}\cdot\eta - g^{1}_{5})=3$ if
 $p$
belongs to the base locus of $\mid D_{12}-H \mid$ and
$h^{0}(K_{C}\cdot\eta - g^{1}_{5})=2$ otherwise. These imply that
for each divisor $D\in \mid g^{1}_{5} \mid$ one has $\dim(<D>)=2$ if
$p$ belongs to the base locus of $\mid D_{12}-H \mid$, and
$\dim(<D>)=3$ otherwise. Moreover consider that since by
$3H-\Delta\sim 2H$, one has
$K_{C}=2g^{1}_{5}\otimes\mathcal{O}(2p)$. Now Proposition
\ref{prepro1} implies that the ruled hypersurface $\cup
_{\acute{D}\in g^{1}_{5}}<\acute{D}>$ is a prym projectivized
tangent cone provided that  $p$ does not belong to the base locus of
$\mid D_{12}-H \mid$.

\noindent In the case $g^{1}_{5}=2H-2\bar{x}-2\bar{y}-2\bar{z}-p$ we
have $K_{C}=2g^{1}_{5}\otimes\mathcal{O}(2p)$ and
$$\begin{array}{ccc}
K_{C}\cdot\eta-g^{1}_{5}&=&\!\!\!\!3H-\Delta
+2H-D_{12}-(2H-2\bar{x}-2\bar{y}-2\bar{z}-p)\\
&=&\!\!\!\!\!\!\!\!\!\!\!\!\!\!\!\!\!\!\!\!\!\!\!\!\!\!\!\!\!\!\!\!\!\!\!\!\!\!4H-\Delta
-D_{12}+p \sim D_{12}-H+p.
\end{array}
$$
Therefore the situation is the same as in the case $g^{1}_{5}=H-p$.

\noindent \textbf{(ii) $g^{1}_{6}$'s on $C$:} Again by
Mumford-Martens theorem and considering that $C$ is not
hyperelliptic nor trigonal and nor a smooth plane
 quintic,
one has $\dim (W^{1}_{6})\leq 6-2-2=2$. The lines in $\mathbf{P}^2$ cut a $g^{2}_{6}$ on
 $C$.
 For each $p, q\in X-\{\bar{x}, \bar{y}, \bar{z}\}$
 the quadrics through $p, q$ and through two of the points $\bar{x}, \bar{y}, \bar{z}$ cut a $g^{1}_{6}$ on
 $X$. Again these pencils are pulled back to $g^{1}_{6}$'s on $C$
via the normalization map.\\
Now similarly as in \textbf{(i)}, the only way for a one dimensional sub-vector space $V$, of
 quadrics in $\mathbf{P}^2$ to cut a $g^{1}_{6}$ on $X$ is that each member of $V$ has to pass
 through two points $p, q$ and through two of the points $\bar{x}, \bar{y}, \bar{z}$. A
 computation similar for $g^{1}_{5}$'s case shows that these are the
 only $g^{1}_{6}$'s on $X$.
 Any $g^{1}_{6}$ on $C$ will be obtained
by pulling back a $g^{1}_{6}$ on $X$ via $i$.
If $g^{1}_{6}=2H-2\bar{x}-2\bar{y}-p-q$ then
$$\begin{array}{ccc}
K_{C}\cdot\eta-g^{1}_{6}&=&\!\!\!\!\!\!\!\!\!\!\!\!\!\!\!\!3H-\Delta+2H-D_{12}-(2H-2\bar{x}-2\bar{y}-p-q)\\
&\sim & D_{12}-2H+2\bar{x}+2\bar{y}+p+q\sim D_{12}-H+p+q-2\bar{z}.
\end{array}
$$
Therefore one has
$$\begin{array}{ccc}
h^{0}(K_{C}\cdot\eta-g^{1}_{6})&=&\!\!\!\!\!\!\!\!\!\!\!\!\!\!\!\!\!\!\!\!\!\!\!\!\!\!\!\!\!\!\!\!\!\!\!\!\!\!\!\!\!\!\!
h^{0}(D_{12}-H+p+q-2\bar{z})\\
&=&h^{0}(D_{12}-H-2\bar{z})= h^{0}(D_{12}-H)-2=1.
\end{array}$$
where the last equality is valid because
$2\bar{z}$ is not contained in the base locus of $\mid D_{12}-H
\mid$. These computations imply that for each $D\in \mid
g^{1}_{6}\mid$ one has $\dim(<D>)=4$ and the union of linear spaces
$<D>\subset\mathbf{P}^5$, when $D$ varies in $g^{1}_{6}$, fill up
all the space $\mathbf{P}^{5}$ and therefore the line bundle
$g^{1}_{6}$ can not give a prym projectivized tangent cone.

\noindent \textbf{(iii) $g^{2}_{5}, g^{2}_{6}$'s on $C$:} The curve
 $C$ can't admit any $g^{2}_{5}$ and $\dim(W^{2}_{6})\leq 0$. Off
course the lines in $\mathbf{P}^2$ cut a $g^{2}_{6}$ on $X$. The
 quadrics passing through the points $\bar{x}, \bar{y}$ and $\bar{z}$ cut a
$g^{2}_{6}$. This is nothing but the $g^{2}_{6}$ cut by the lines
 in $\mathbf{P}^2$.  Again any $g^{2}_{6}$ on $C$ will be obtained
 by pulling back a $g^{2}_{6}$ on $X$ via $i$, the normalization map.
 }
\end{exam}
%%%%%%%%%%%%%%%%%%%%%%%%%%%%%%%%%%%%%%%%%%%%
%%%%%%%%%%%%%%%%%%%%%%%%%%%%%%%%%%%%%%%%%%%%%%%%
%%%%%%%%%%%%%%%%%%%%%%%%%%%%%%%%%%%%%%%%%%%%%%%%%%%
\noindent As a byproduct of the computations just have been done,
Theorem \ref{theorem3} and Proposition \ref{prepro1}, one has the
following:
\begin{theorem}{
The space of quadrics containing the $K_{C}\cdot\eta$-model of
%general tetragonal curve of genus $7$?,
$C$ is generated by Prym projectivized tangent cones at double
points of $\mathbb{E}(\pi_{\eta})$. Precisely the set of quadrics
$$\{Q_{g^{1}_{5}}\mid g^{1}_{5}\in W^{1}_{5}\}$$
which consists of a subset of projectivized tangent quadrics of rank equal or less than $4$, 
generate the space of quadrics
containing $C_{\eta}$ in $\mathbf{P}^{5}$.}
\end{theorem}
\begin{proof}{
There exists a map $\Phi$ defined by
$$
\begin{array}{ccc}
\Phi: &W^{1}_{5}\rightarrow &\mathbf{P}(\mathcal{I}_{2}(C))\cong \mathbf{P}^{2}\\
\qquad
\Phi(g^{1}_{5})=\!\!\!\!\!\!&Q_{g^{1}_{5}}&\!\!\!\!\!\!\!\!:=\cup_{D\in
g^{1}_{5}}<D>
\end{array}
$$
Consider that $\Phi$ is an embedding and by our computations, its
image is contained in the locus of projectivized tangent cone
quadrics inside $\mathbf{P}(\mathcal{I}_{2}(C))$. Moreover consider
again by example \ref{example1} that $W^{1}_{5}$ consists of two
copies both are birational to $C$ itself. Therefore the image of the
map $\Phi$ is non-degenerate. These imply that the linear span of
projectivized tangent cones is $\mathbf{P}(\mathcal{I}_{2}(C))$. }
\end{proof}
\begin{remark}
Based on Debarre's work in \cite{D1}, the Prym-Torelli map
$$\mathcal{P}:\mathcal{R}_{g}\rightarrow \mathcal{A}_{g-1}$$
is generically injective. This fails in the non-generic locus'
because of well known reasons. In fact Donagi's tetragonal
construction as well as generalized tetragonal construction
introduced in \cite{I-L} imply that, for an etale double covering
$\bar{Y}\rightarrow Y$, of a tetragonal (generalized tetragonal)
curve $Y$, there exist two another un-ramified double coverings
having Prymians isomorphic to that of $\bar{Y}\rightarrow Y$. This
implies non-injectivity of the Prym-Torelli map in the locus of
double coverings of tetragonal (generalized tetragonal) curves in
$\mathcal{R}_{g}$.

On the other hand as it has been noticed by H. Lange and E. Sernesi
for the injectivity of the Torelli map in \cite{L-S}, an effective
strategy to deal with the injectivity of the Prym-Torelli map seems
to consist of two main steps.
The first step is to show that for a
given unramified double cover $\pi:\bar{X}\rightarrow X$,
the
projectivized tangent cones at double points of the Prym-theta
divisor of the Prym variety generate the space of quadrics through
the Prym-canonical model of the double covering.
 This step has been
done by Debarre in \cite{D1} for curves varying in an open subset of
$\mathcal{R}_{g}$, as we already noticed.

The second main step consists in proving that the quadrics through
the Prym-canonical model of $\pi:\bar{X}\rightarrow X$ in the
projective space of the Prym-canonical differential forms, cut the
Prym-canonical model. This step has been also proved not only for
general Prym-canonical curves by Debarre in \cite{D1}, but also for
general tetragonal curves of genus at least $11$ by him in \cite{D}.
Meanwhile, H. Lange and E. Sernesi have done this step for
unramified double coverings of curves $X$ with $\Cliff(X)\geq 3$ in
\cite{L-S}.

Consider now that Debarre's work in \cite{D}, together with
non-injectivity of Prym-Torelli map in tetragonal(generalized
tetragonal) locus, imply that the projectivized tangent cones at
double points of the Prym-theta divisor of the Prym variety of a
general tetragonal curve of genus at least $11$, or that of a
generalized tetragonal curve, do not generate the space of quadrics
through the Prym-canonical model of an un-ramified double coverings
of such a curve.

These however won't give any information about validity or
invalidity of the first step for an arbitrary tetragonal curve of
genus $g\leq11$, as it is concluded from our work that the first
step remains valid for an etale double cover of the curve $C$. This
as well proves that the quadrics through the Prym-canonical model of
an arbitrary etale double covering of the curve $C$, does not cut
its Prym-canonical model.
%\mathfrak{}?

\end{remark}

\textbf{Acknowledgements:} Professor E. Sernesi read parts of an
early version of this manuscript and suggested some corrections. He
moreover answered my various questions patiently. I am grateful to
him and I express my hearty thanks to him. The curve $C$ in the
paper, has been addressed to me by
 professor A. Verra and professor G. Farkas when I was visiting the Rome Tre University on fall of 2011. 
 I express my deep gratitude for
 professor A. Verra and professor G. Farkas for this and another hints that I received from them.
 %}
%\end{proof}?

\end{document}